\documentclass[12pt, reqno]{amsart}

\usepackage[margin=1in]{geometry} 

\usepackage{amssymb,upref}
\usepackage{amsmath, amsthm}
\usepackage{enumerate, graphicx, float}
\usepackage{relsize, mathrsfs, upgreek, bbm, eufrak}

\usepackage{color}
\usepackage{esint}

\theoremstyle{plain}

\newtheorem{thm}{Theorem}[section]
\newtheorem{lem}[thm]{Lemma}
\newtheorem{coro}[thm]{Corollary}

\theoremstyle{remark}

\newtheorem{remark}[thm]{Remark}
\numberwithin{equation}{section}

\def\XXint#1#2#3{{\setbox0=\hbox{$#1{#2#3}{\int}$ }
\vcenter{\hbox{$#2#3$ }}\kern-.6\wd0}}

\newcommand{\bee}{\begin{equation}}
\newcommand{\eee}{\end{equation}}
\newcommand{\be}{\begin{equation*}}
\newcommand{\ee}{\end{equation*}}

\newcommand{\eps}{\varepsilon}

\newcommand{\na}{\mathbb{N}}
\newcommand{\re}{\mathbb{R}}
\newcommand{\rn}{\mathbb{R}^n}

\newcommand{\norm}[2]{\|#1\|_{#2}}

\newcommand{\diam}[1]{{\rm diam}(#1)}
\newcommand{\supp}[1]{{\rm supp}(#1)}
\newcommand{\dist}[1]{{\rm d}(#1)}

\newcommand{\dx}{\, dx}

\newcommand{\Ainf}{A_\infty(\rn)}

\newcommand{\Muinf}{{\rm{Mu}_\infty}}

\newcommand{\Mu}{{\rm{Mu}}}

\DeclareMathOperator*{\esssup}{ess\,sup\,}

\DeclareMathOperator*{\essinf}{ess\,inf\,}

\makeatletter
\@namedef{subjclassname@2020}{\textup{2020} Mathematics Subject Classification}
\makeatother

\begin{document}



\address{Diego Maldonado, Kansas State University, Department of Mathematics. 138 Cardwell Hall, Manhattan, KS-66506, USA.} \email{dmaldona@ksu.edu}

\address{Javier Soria, Departamento de An\'alisis Matem\'atico y Matem\'atica Aplicada, Fa\-cul\-tad de Matem\'aticas, Universidad Complutense de Madrid, Plaza de Ciencias 3, 28040 Madrid, Spain and ICMAT.}
\email{javier.soria@ucm.es}

\title[]{Weighted estimates for fractional integrals with Distances to Bounded Median Porous Sets and applications to Hardy--Sobolev Inequalities}
\author[Diego Maldonado and Javier Soria]{Diego Maldonado$^{*}$ and Javier Soria$^{**}$}

\thanks{$^{*}$Partially supported by Simons Foundation grant MPS-TSM-00007229.\,$^{**}$Partially supported by grants PID2024-155917NB-I00, and CEX-2023-001347-S, funded by MCIN/AEI/ 10.13039/501100011033, and Grupo UCM-970966.}

\date{\today}

\begin{abstract} Weighted estimates for the fractional integral operator $I_\alpha$ are established and subsequently applied to derive corresponding Hardy--Sobolev inequalities. The weights are constructed from distance functions to bounded median porous sets and possess mixed homogeneity, which enables us to extend earlier results obtained for porous sets to a significantly broader class of geometries.
\end{abstract}

\maketitle

\section{Introduction and main results}\label{sec:main:results}

For $1 < p < n$ the classical Hardy or weighted Hardy-Sobolev inequality establishes the existence of $C_{p,n} >0$ such that 
\begin{equation}\label{Hardy:ineq:p:p}
\int_{\rn}  |f(x)|^p |x|^{-p} \, dx \leq C_{p,n} \int_{\rn} |\nabla f(x)|^p \, dx, \quad \forall f \in C_0^1(\rn),
\end{equation}
where $C_0^1(\rn)$ denotes the class of continuously differentiable functions vanishing at $\infty$. More generally, for $1 < p \leq q \leq np/(n-p) < \infty$, H. Egnell \cite[Lemma~7]{Eg1} (see also V. G. Maz'ja's \cite[Section~1.3.1]{Maz}) proved the inequality 
\begin{equation}\label{EM:ineq:p:q}
\left(\int_{\rn}  |f(x)|^q |x|^{\frac{q}{p}(n-p) -n} \, dx \right)^{1/q}\leq C_{n,p,q} \left(\int_{\rn} |\nabla f(x)|^p \, dx\right)^{1/p}, \quad \forall f \in C_0^1(\rn).
\end{equation}
If $q=p$ then \eqref{EM:ineq:p:q} reduces to the Hardy inequality \eqref{Hardy:ineq:p:p}. Given a set $E \subset \rn$, the distance function to $E$ is defined for $x \in \rn$ as $\dist{x, E}:=\inf\{|x-y|: y \in E\}.$ When $|x| = \dist{x,\{0\}}$ is replaced with the distance to a (nonempty) closed set $E \subset \rn$, it has been proved by J.~Lehrb\"ack and A.~V\"ah\"akangas in \cite[Theorem~1.1]{LeVa} that given $1 \leq p \leq q < np/(n-p) < \infty$, the Hardy-type inequality
\begin{equation}\label{Sob:porous}
\left(\int_{\rn}  |f(x)|^q \, \dist{x,E}^{\frac{q}{p}(n-p) -n} \, dx \right)^{1/q} \leq C \left( \int_{\rn} |\nabla f(x)|^p \, dx \right)^{1/p},\quad \forall f \in C_0^1(\rn),
\end{equation}
holds true if and only if $\dim_A(E) < q(n-p)/p$ (here $\dim_A(E)$ denotes the  Assouad dimension of $E$, see Section \ref{sec:Assouad}). 

Back to powers of $|x|$, extensions of \eqref{EM:ineq:p:q} include the classical Caffarelli-Kohn-Nirenberg inequalities proved in \cite{CKN1}: given $1 < p \leq q \leq np/(n-p) < \infty$ and $\beta > -(n-p)$ there exists $C=C(p, q, \beta, n) >0$ such that
\begin{equation}\label{CKN:ineq:2:q}
\left(\int_{\rn}  |f(x)|^q |x|^{q(n-p+\beta)/p -n} \, dx \right)^{1/q} \leq C \left(\int_{\rn} |\nabla f(x)|^p |x|^{\beta}\, dx \right)^{1/p}, \quad \forall f \in C_0^1(\rn).
\end{equation}
Now, by considering $\dist{x, E}$ instead of $|x|$ in \eqref{CKN:ineq:2:q} for a closed set $E \subset \rn$, $1 \leq p, q < \infty$, and $\beta \geq 0$ with $q(n-p+\beta)/p \neq n$, from \cite[Theorem 6.1]{LeVa}  the inequality 
\begin{equation}\label{Hardy:ineq:p:q:P}
\left(\int_{\rn} |f(x)|^q \,  \dist{x,E}^{q(n-p+\beta)/p -n} \dx \right)^{1/q}\leq C  \left(\int_{\rn} |\nabla f(x)|^p \, \dist{x,E}^{\beta} \dx \right)^{1/p},
\end{equation}
for some constant $C > 0$ and every $f \in C_0^1(\rn)$, implies $\dim_A(E) < q(n - p + \beta)/p.$  In particular, if $q(n - p + \beta)/p < n$, then $\dim_A(E)  < n$.

Let us mention that in  \cite[Theorem 4.1]{DIL+}, B. Dyda, L. Ihnatsyeva, J. Lehrb\"ack, H. Tuominen, and A. V\"ah\"akangas proved that for $1 < p \leq q \leq pn/(n-p) < \infty$ and $\beta \in \re$, the condition
\begin{equation}\label{DIL+:sufficient}
\dim_A(E) <  \min\{q(n - p + \beta)/p, n - \beta/(p-1)\}
\end{equation}
is sufficient for \eqref{Hardy:ineq:p:q:P} to hold. By setting $\underline{{\rm{codim}}}_A(E):= n -{\rm{dim}}_A(E)$, the inequality \eqref{DIL+:sufficient} can be recast as 
\begin{equation}\label{DIL+:sufficient:co}
\underline{{\rm{codim}}}_A(E) > \min\{n -q(n - p + \beta)/p, \beta/(p-1)\}.
\end{equation}
A set $E \subset \rn$ with $\dim_A(E)  < n$ (that is, $\underline{{\rm{codim}}}_A(E) >0$) is called a \emph{porous set} (see Section~\ref{secc:porous:sets} for equivalent definitions). The class of porous sets includes that of $\lambda$-regular sets for $0 < \lambda < n$, where a closed set $E \subset \rn$ is called \emph{$\lambda$-regular} if there exists $C \geq 1$ such that
$$
C^{-1}r^\lambda \leq \mathcal{H}^\lambda(E\cap B(x, r)) \leq C r^\lambda, \quad \forall x \in E, 0 < r < \diam{E},
$$
where $\diam{E}:=\sup\{|x-y|: x, y \in E\}$  indicates the diameter of $E$ and $\mathcal{H}^\lambda$ stands for the $\lambda$-dimensional Hausdorff measure. If $E$ is $\lambda$-regular, ${\rm{dim}}_A(E) = \lambda$ (see for instance \cite[Theorem~10.21]{KLV}) which makes $E$ porous when $\lambda < n$.

The fact that the porosity of $E$ is a necessary condition for \eqref{Sob:porous} and \eqref{Hardy:ineq:p:q:P} naturally leads one to ask whether corresponding inequalities exist in the case of nonporous sets. To answer this, let us recall the notions of weak and median porosity  introduced in \cite{ALMV} and \cite{Pa-UT}  (see Section~\ref{secc:distance:weight:porosity} for details). Following T.~Anderson, J.~Lehrb\"ack,  C.~Mudarra, and A.~V\"ah\"akangas in \cite{ALMV} a nonempty set $E \subset \rn$  is  \emph{weakly porous} if there exists $\alpha > 0$ such that $\dist{\cdot, E}^{-\alpha}$ belongs to the Muckenhoupt class $A_1(\rn)$ (the $A_p(\rn)$ classes are reviewed in Section \ref{sec:Ap:weights}). In turn, M.~Pasquariello and I.~Uriarte--Tuero in \cite{Pa-UT} and, independently,  I.~G\'omez-Vargas in \cite{IGV26} characterized the sets $E \subset \rn$ for which there exists $\alpha > 0$ with $\dist{\cdot, E}^{-\alpha} \in A_p(\rn)$ for some $1 \leq p < \infty$. In \cite{Pa-UT}, such sets are called \emph{median porous sets}. As it turns out, 
$$
{\text{porosity }} \Rightarrow {\text{weak porosity }} \Rightarrow {\text{median porosity}}
$$
with the converses to the above implications being false in general.

As further described in Section \ref{secc:distance:weight:porosity}, the properties of weak and median porosity for $E$ are quantified by a family of nonnegative indices $\{{\rm{Mu}}_p(E)\}_{1 \leq p \leq \infty}$ (with $\rm{Mu}$ after Muckenhoupt) which come to refine the notion of the codimension $\underline{{\rm{codim}}}_A$ in the sense that $E$ is weakly porous if and only if ${\rm{Mu}}_1 (E) >0$ (with ${\rm{Mu}}_1 (E) = \underline{{\rm{codim}}}_A(E)$ whenever $E$ is porous); whereas $E$ is median porous  if and only if ${\rm{Mu}}_\infty (E) >0$ (with ${\rm{Mu}}_p (E) = \underline{{\rm{codim}}}_A(E)$ for every $1 \leq p \leq \infty$ whenever $E$ is weakly porous). For example, in \cite[Section~7]{ALMV} it is proved that for $n \in \na$ and $\gamma > 0$, the bounded set
$$
E:= \bigcup\limits_{j \in \na} \partial B(0, j^{-\gamma}) \bigcup \{0\}  \subset \rn
$$
is weakly porous, but not porous, and it satisfies
$$
\dim_A(E) = n, \quad  \overline{\rm{dim}_M}(E)  = \max\left\{n-1, \frac{n}{1+\gamma}\right\},  \quad \Mu_1  =  n - \overline{\rm{dim}_M}(E).
$$
Here $\overline{\rm{dim}_M}(E)$ stands for the upper Minkowski dimension of $E \subset \rn$ (see Section \ref{secc:dim:M}). In addition, in \cite[Section~9]{Pa-UT}, the authors showed that for $0 < \gamma < 1$, the subset of the real line
$
E_\gamma :=\{ \pm m^\gamma: m \in \na\}
$
is a median porous set which is not weakly porous.

In the context of the weighted Hardy inequality \eqref{Hardy:ineq:p:q:P}, the first result to ``break the porosity barrier'' (that is, allowing for $\underline{{\rm{codim}}}_A(E) =0$) comes from \cite[Theorem 10.8(a)]{Pa-UT} where for $n \geq 2$, $1 < p < n$, and $q=p^*:=np/(n-p)$, the conditions 
\begin{equation}\label{Pa-UT:Mu}
{\rm{Mu}}_\infty (E) > n - q(n - p + \beta)/p \quad \text{ and } \quad -{\rm{Mu}}_p(E) < \beta < 0 
\end{equation}
are shown to imply \eqref{Hardy:ineq:p:q:P}; thus, providing a sufficient condition, counterpart to the sufficient condition \eqref{DIL+:sufficient:co}, within the median porous case.

For $1 < p \leq q < \infty$ and $0 < \alpha < n$ let us define
\begin{equation}\label{def:alpha0}
\Theta_{p,q}(\alpha):= \alpha + n \left( \frac{1}{q} - \frac{1}{p}\right).
\end{equation}
The cases $\Theta_{p,q}(\alpha) =0$ (i.e., $q = np/(n-\alpha p)$) and $\Theta_{p,q}(\alpha) >0$ (i.e., $q < np/(n-\alpha p)$) are referred to as the \emph{critical} and \emph{subcritical} cases, respectively, for the indices $p, q, \alpha$. The conditions from \eqref{Pa-UT:Mu} correspond then to the critical case with $\alpha =1$. Given a median porous set $E$, the weights $u_{\gamma_1, \gamma_2}$ that we will consider are defined for $x \in \rn \setminus E$ as
\begin{equation}\label{def:u:gamma:1:2:c}
u_{\gamma_1, \gamma_2}(x) :=   \left\{
      \begin{array}{cl}
       \max\left\{ \left(\frac{\dist{x, E}}{\diam{E}}\right)^{\gamma_1}, \left(\frac{\dist{x, E}}{\diam{E}}\right)^{\gamma_2} \right\}    &  \text{if } \gamma_1 < \gamma_2 \\
	\min \left\{ \left(\frac{\dist{x, E}}{\diam{E}}\right)^{\gamma_1}, \left(\frac{\dist{x, E}}{\diam{E}}\right)^{\gamma_2} \right\}    &  \text{if } \gamma_1 \geq \gamma_2,
\end{array}
    \right. 
\end{equation}
for suitable choices of exponents $\gamma_1, \gamma_2 \in \re$. Since $u_{\gamma_1, \gamma_2}$ can be rewritten as
\begin{equation}\label{def:u:F}
u_{\gamma_1, \gamma_2}(x) :=   \left\{
      \begin{array}{cl}
       \left(\frac{\dist{x, E}}{\diam{E}}\right)^{\gamma_1}  &  \text{if } \dist{x, E} < \diam{E} \\
	\left(\frac{\dist{x, E}}{\diam{E}}\right)^{\gamma_2}  &  \text{if } \dist{x, E} \geq \diam{E},
\end{array}
    \right. 
\end{equation}
the exponents $\gamma_1, \gamma_2$ equip $u_{\gamma_1, \gamma_2}$ with a ``mixed'' homogeneity;  namely, $\gamma_1$ in the local case $\dist{\cdot, E} < \diam{E}$ and $\gamma_2$ when $\dist{\cdot, E} \geq \diam{E}$. This mixed homogeneity allows for a number of Hardy-Sobolev inequalities, including counterparts to \eqref{Sob:porous} and \eqref{Hardy:ineq:p:q:P}, in both the critical and subcritical cases in the absence of porosity  (see Remark \ref{rmk:mixed:homogeneity} below).

Given a nonempty set $E \subset \rn$, let
$$
I_\infty(E):=\{\beta \in \re: \dist{\cdot, E}^{-\beta} \in \Ainf\},
$$ 
hence, $E$ is median porous if and only if  $I_\infty(E) \ne \emptyset$. Also, for $0 < \alpha < n$ the fractional integral operator $I_\alpha$ is defined as
$$
I_\alpha(f)(x):= \int_{\rn} \frac{f(y)}{|x-y|^{n-\alpha}}\, dy. 
$$
We are now in position to state our main results on weighted estimates for fractional integrals based on distance functions to bounded, median porous sets.

\begin{thm}\label{thm:indices:SW} Fix $1 < p \leq q < \infty$ and $0 < \alpha < n$ with $\Theta_{p,q}(\alpha)> 0$ (i.e., the subcritical case). Let $E \subset \rn$ with $0 < \diam{E} < \infty$ be a median porous set. Then, for any exponents $\gamma_1, \gamma_2, \delta_1, \delta_2 \in \re$ verifying 
\begin{equation}\label{cond:gamma:1:2}
- \gamma_j \in I_\infty(E) \quad \text{for } j=1,2,
\end{equation}
\begin{equation}\label{cond:delta:1:2}
\delta_j p'/p \in I_\infty(E) \quad \text{for } j=1,2,
\end{equation}
\begin{equation}\label{cond:s:dimB:n}
\overline{\rm{dim}_M}(E) - n < \frac{n}{\Theta_{p,q}(\alpha)}\left(\frac{\gamma_1}{q} - \frac{\delta_1}{p}\right)  < 0,
\end{equation}
and 
\begin{equation}\label{cond:gamma2:delta2}
\left(\frac{\delta_2}{p}-\frac{\gamma_2}{q} \right)  = \Theta_{p,q}(\alpha),
\end{equation}
there exists $C >0$, depending only on $\gamma_1, \gamma_2, \delta_1, \delta_2$, $p$, $q$, $\alpha$, $n$, and $E$,  such that  
\begin{equation}\label{bound:T:alpha:SW} 
\left(\int_{\rn} |I_\alpha(f)(x)|^q  u_{\gamma_1, \gamma_2}(x) \dx \right)^{1/q}\leq C \left(\int_{\rn} |f(x)|^p \, u_{\delta_1, \delta_2}(x) \dx \right)^{1/p},
\end{equation}
for every measurable $f$.
\end{thm}

\begin{remark}\label{rmk:Hardy:beyond:porosity} Regarding the feasibility of the condition \eqref{cond:s:dimB:n}, by Lemma \ref{lemma:Muinf<n-dimM} below we have $ \Muinf(E) \leq n - \overline{{\rm{dim}}}_M(E)$ for every nonempty bounded set $E \subset \rn$. If, in addition, $E$ is median porous (i.e., $\Muinf(E) >0$) it follows that $\overline{{\rm{dim}}}_M(E) < n$ and therefore indices $p, q, \delta_1, \gamma_1$ satisfying \eqref{cond:s:dimB:n} can always be chosen. 
\end{remark}

The case $\alpha =1$ in Theorem \ref{thm:indices:SW} yields the following two-weight Hardy-Sobolev inequality. 

\begin{coro}\label{coro:Hardy:CKN:MP} Fix $n > 1$ and $1 < p \leq q < \infty$ with 
$$
\Theta_{p,q}(1):= 1 + n (1/q-1/p)> 0,
$$
that is, the subcritical case $q < np/(n-p).$  Let $E \subset \rn$ with $0 < \diam{E} < \infty$ be a median porous set. Then, given exponents $\gamma_1, \gamma_2, \delta_1, \delta_2 \in \re$ satisfying
 \begin{align*}
      -\gamma_j, \delta_j p'/p & \in I_\infty(E) \quad \text{for } j=1,2,\\
      \overline{\rm{dim}_M}(E) -n  < \frac{n}{\Theta_{p,q}(1)}\left(\frac{\gamma_1}{q} - \frac{\delta_1}{p}\right) &  < 0, \\
	\left(\frac{\delta_2}{p}-\frac{\gamma_2}{q} \right) & = \Theta_{p,q}(1),    
\end{align*}
we have that $u_{\gamma_1, \gamma_2}, u_{\delta_1, \delta_2}^{-p'/p} \in \Ainf$ and that there exists $C >0$, depending only on $p$, $q$, $n$, $\gamma_1, \gamma_2, \delta_1, \delta_2$, and $E$ such that
$$
\left(\int_{\rn} |f(x)|^q  \, u_{\gamma_1, \gamma_2}(x) \dx \right)^{1/q}\leq C \left(\int_{\rn} |\nabla f(x)|^p \, u_{\delta_1, \delta_2}(x) \dx \right)^{1/p},   \quad \forall f \in C_0^1(\rn).
$$ 
\end{coro}

The case  $\delta_1= \delta_2  =0$ in Corollary \ref{coro:Hardy:CKN:MP} provides a counterpart to \eqref{Sob:porous} when $E$ is bounded and median porous. More precisely,

\begin{coro}\label{coro:v=1:p<q<p*} Fix $1 < p  < n$ and $p\leq q < np/(n-p)$. Let $E \subset \rn$ be a median porous set with $0 < \diam{E} < \infty$. Then, for any exponents $\gamma_1, \gamma_2$ satisfying 
\begin{align}\label{cond:gamma:1:2:delta=0:a}
0 > \gamma_1 >& \left(\frac{q}{p}(n-p) -n \right)\left(1- \frac{\overline{\rm{dim}_M}(E)}{n} \right),\\\label{cond:gamma:1:2:delta=0:b}
0 >  \gamma_2 =& \frac{q}{p}(n-p) -n > - \Muinf(E),
\end{align}
there exists  $C > 0$ depending only on $\gamma_1, \gamma_2$ , $p$, $q$, $n$, and $E$ such that
\begin{equation}\label{I1:v=1:p<q<p*} 
\left(\int_{\rn} |I_1(f)(x)|^q  \min \left\{\dist{x, E}^{\gamma_1}, \dist{x, E}^{\gamma_2} \right\} \dx \right)^{1/q}\leq C \left(\int_{\rn} |f(x)|^p \dx \right)^{1/p},
\end{equation}
for every $f \in L^p(\rn)$. Consequently, the weighted Hardy-Sobolev inequality
\begin{equation}\label{Hardy:p:q:E:q<p*}
\left(\int_{\rn} |f(x)|^q  \min \left\{\dist{x, E}^{\gamma_1}, \dist{x, E}^{\gamma_2} \right\} \dx \right)^{1/q}\leq C \left(\int_{\rn} |\nabla f(x)|^p \dx \right)^{1/p}
\end{equation}
holds true for every $f \in C_0^1(\rn)$.
\end{coro}

\begin{remark}\label{rmk:mixed:homogeneity} The Hardy-type inequality \eqref{Hardy:p:q:E:q<p*}, with its mixed $(\gamma_1, \gamma_2)$-homogeneity, comes then as a substitute for \eqref{Sob:porous} in the absence of porosity. Indeed, for $\gamma_1, \gamma_2$ as in \eqref{cond:gamma:1:2:delta=0:a} and \eqref{cond:gamma:1:2:delta=0:b} we have $\gamma_2 < \gamma_1 < 0$ and the inequality $\dist{x, E} \leq 1$ is equivalent to $\dist{x, E}^{\gamma_1} \leq \dist{x, E}^{\gamma_2}$. Hence if $\supp{f}\subset \{x \in \rn: \dist{x, E} > 1\}$, then \eqref{Hardy:p:q:E:q<p*} reduces to \eqref{Sob:porous} (since $\gamma_2$ in \eqref{cond:gamma:1:2:delta=0:b} equals the corresponding exponent in 
\eqref{Sob:porous}), whereas if $\supp{f}\subset \{x \in \rn: \dist{x, E} \leq 1\}$, then \eqref{Hardy:p:q:E:q<p*} means 
$$
\left(\int_{\rn} |f(x)|^q \dist{x, E}^{\gamma_1} \dx \right)^{1/q}\leq C \left(\int_{\rn} |\nabla f(x)|^p \dx \right)^{1/p},
$$
with $\gamma_1 =  \left(1- \frac{\overline{\rm{dim}_M}(E)}{n}\right) \gamma_2.$
\end{remark}

From the case $p=q$ and $\alpha =1$ (hence, $\Theta_{p,q}(1)=1$) in Theorem \ref{thm:indices:SW} we obtain

\begin{thm}\label{thm:gamma:1:2:delta:1:2:p=q} Fix $n > 1$ and $1 < p < \infty$. Let $E \subset \rn$ be a median porous set with $0 < \diam{E} < \infty$. Then, for any exponents $\gamma_1, \gamma_2, \delta_1, \delta_2$ satisfying 
\begin{equation}\label{cond:gamma:1:2:delta:1:2:p=q}
\left\{
      \begin{array}{rl}
     -\gamma_j, \delta_j p'/p &\in I_\infty(E) \quad \text{for } j=1,2\\[2 mm]
      \overline{\rm{dim}_M}(E) - n & < \frac{n}{p} ( \gamma_1 - \delta_1) <0  \\[2 mm]
	\gamma_2 - \delta_2 & = -p,
\end{array}
    \right. 
\end{equation} 
we have $u_{\gamma_1, \gamma_2}, u_{\delta_1, \delta_2}^{-p'/p} \in \Ainf$ and there exists $C >0$, depending only on $\gamma_1, \gamma_2, \delta_1, \delta_2$, $p$, $n$, and $E$,  such that  
\begin{equation}\label{gamma:1:2:delta:1:2:p=q}
\int_{\rn} |I_1(f)(x)|^p  u_{\gamma_1,\gamma_2}(x) \dx \leq C \int_{\rn} |f(x)|^p v_{\delta_1,\delta_2}(x) \dx,
\end{equation}
for every measurable function $f$. Consequently,  
\begin{equation}\label{HS:Theta1=1}
\int_{\rn} |f(x)|^p  u_{\gamma_1,\gamma_2}(x) \dx \leq C \int_{\rn} |\nabla f(x)|^p v_{\delta_1,\delta_2}(x) \dx, \quad \forall f \in C_0^1(\rn).
\end{equation}
\end{thm}

The next theorem establishes two-weight inequalities for $I_\alpha$ in the critical case $\Theta_{p,q}(\alpha)=0$, that is, $q=np/(n- \alpha p)$.

\begin{thm}\label{thm:indices:SW:q=p*} Fix $0 < \alpha < n$, $1 < p < n/\alpha$, and set $q:=np/(n- \alpha p)$. Let $E \subset \rn$ be a median porous set with $0 < \diam{E} < \infty$. Then, for any exponents $\gamma_1, \gamma_2, \delta_1, \delta_2 \in \re$ verifying \eqref{cond:gamma:1:2}, \eqref{cond:delta:1:2}, and
\begin{align}\label{gamma1:delta1:q=p*}
\gamma_1 & \geq \frac{q }{p}\delta_1, \\\label{gamma2:delta2:q=p*}
\gamma_2 & \leq \frac{q }{p}\delta_2,
\end{align}
we have $u_{\gamma_1, \gamma_2}, u_{\delta_1, \delta_2}^{-p'/p} \in \Ainf$ and there exists $C >0$, depending only on $\gamma_1, \gamma_2, \delta_1, \delta_2$, $p$, $\alpha$, $n$, and $E$,  such that  
\begin{equation}\label{bound:T:alpha:SW:2} 
\left(\int_{\rn} |I_\alpha(f)(x)|^q  u_{\gamma_1, \gamma_2}(x) \dx \right)^{1/q}\leq C \left(\int_{\rn} |f(x)|^p \, u_{\delta_1, \delta_2}(x) \dx \right)^{1/p},
\end{equation}
for every measurable $f$. 
\end{thm}

As a particular case of Theorem~\ref{thm:indices:SW:q=p*} we recover \cite[Theorem~10.4]{Pa-UT} in the case of bounded sets. Namely, 
 
\begin{coro}\label{coro:beta:1:2:critical} Fix $0 \leq \alpha < n$, $1 < p < n/\alpha$, and set $q=:np/(n- \alpha p)$. Let $E \subset \rn$ be a median porous set with $0 < \diam{E} < \infty$. Consider exponents $\beta_1, \beta_2 \in \re$ satisfying
\begin{equation}\label{coro:beta:1:2:a}
- \beta_2 = - \frac{q}{p} \beta_1\in I_\infty(E)
\end{equation}
and
\begin{equation}\label{coro:beta:2:a}
\quad \frac{\beta_1}{p-1} \in I_\infty(E).
\end{equation} 
Then, there exists $C > 0$ depending only on $\beta_1, \beta_2,$ $\alpha$, $p$, $n$,  and $E$, such that
\begin{equation*}
\left(\int_{\rn} |I_\alpha(f)(x)|^q \dist{x, E}^{\beta_2} \dx \right)^{1/q} \leq C \left(\int_{\rn} |f(x)|^p  \dist{x, E}^{\beta_1} \dx \right)^{1/p},
\end{equation*}
for every measurable function $f$. 
\end{coro}

As another consequence of Theorem \ref{thm:indices:SW:q=p*} we obtain the following two-weight Hardy-Sobolev inequality in the critical case $q = np/(n-p)$ (that is, $\Theta_{p,q}(1)=0$).

\begin{thm}\label{thm:app:q=p*:E}  Fix $1 < p < n$ and set $q:=np/(n- p)$. Let $E \subset \rn$ be a median porous set with $0 < \diam{E} < \infty$. Then, for any exponents $\gamma_1, \gamma_2, \delta_1, \delta_2 \in \re$ verifying 
\begin{align*}
-\gamma_j, \delta_j p'/p & \in I_\infty(E) \quad \text{for } j=1,2,\\
\gamma_1 & \geq \frac{n}{(n-p)}\delta_1 \\
\gamma_2 & \leq \frac{n }{(n-p)}\delta_2,
\end{align*}
we have $u_{\gamma_1, \gamma_2}, u_{\delta_1, \delta_2}^{-p'/p} \in \Ainf$ and there exists $C >0$ depending only on $p$, $n$, $\gamma_1, \gamma_2, \delta_1, \delta_2$, and $E$, such that 
$$
\left(\int_{\rn} |f(x)|^q  u_{\gamma_1, \gamma_2}(x) \dx \right)^{1/q}\leq C \left(\int_{\rn} |\nabla f(x)|^p \, u_{\delta_1, \delta_2}(x) \dx \right)^{1/p},   \quad \forall f \in C_0^1(\rn).
$$
\end{thm}

\subsection{A blow-up rate for the constant $C$ in \eqref{Hardy:ineq:p:q:P} under weak porosity}  Regarding the blow-up rate of the constant $C >0$ from  \eqref{Hardy:ineq:p:q:P}, in the case of a bounded, weakly porous set and as the left integral in \eqref{Hardy:ineq:p:q:P} is performed increasingly closer to $E$ we have

\begin{thm}\label{thm:Hardy:ineq:p:q:WP:2}  Let $n > 1$ and $1 < p  \leq q < n/\Theta_{p,q}(1)$ with $\Theta_{p,q}(1) :=1 + n (1/q-1/p) > 0$. Fix a  weakly porous set $E \subset \rn$  with $0 < \diam{E} < \infty$ and $\beta > 0$ with $\beta p'/p < {\rm{Mu}}_1(E)$. Then, there exists $C > 0$ such that
\begin{equation*}
\left(\int_{D(\kappa)} |f(x)|^q   d_E(x)^{q(n - p + \beta)/p - n} \dx \right)^{1/q}\leq C \kappa^{-\Theta_{p,q}(1)} \left(\int_{\rn} |\nabla f(x)|^p d_E(x)^{\beta} \dx \right)^{1/p}
\end{equation*}
for every $f \in C_0^1(\rn)$ and $0 < \kappa < 1$, with 
$
D(\kappa) := \{x \in \rn: \dist{x, E} \geq \kappa \, \diam{E}\}.
$
\end{thm}

The rest of the article is organized as follows: Section \ref{secc:prelim} contains preliminary material on Muckenhoupt weights as well as on various notions of dimension and porosity of sets in $\rn$. Section  \ref{sec:gE} introduces a function $g_E$ associated to a bounded set $E \subset \rn$ and characterizes the condition $g_E \in L^{1, \infty}(\rn)$ in terms of measure-theoretic properties of $E$. Section \ref{sec:proofs:dist:E} contains the proofs of Theorems \ref{thm:indices:SW}, \ref{thm:gamma:1:2:delta:1:2:p=q},  \ref{thm:indices:SW:q=p*}, \ref{thm:app:q=p*:E} and that of Corollaries \ref{coro:v=1:p<q<p*} and \ref{coro:beta:1:2:critical}. These proofs heavily rely on the aforementioned function $g_E$ to identify pairs $(u_{\gamma_1, \gamma_2},u_{\delta_1, \delta_2})$ satisfying either the condition \eqref{u:1/q:C:v:1/p} from Theorem \ref{thm:SW:Ainf:p:q:alpha0=0} or the condition \eqref{cond:u:v:Theta>0} from Theorem \ref{thm:SW:Ainf:p:q}, when $E \subset \rn$ is a bounded, median porous set. Finally, Section \ref{sec:proof:WP:2} is devoted to the proof of Theorem \ref{thm:Hardy:ineq:p:q:WP:2}.  

\section{Preliminaries}\label{secc:prelim}

\subsection{Muckenhoupt weights}\label{sec:Ap:weights} 
Recall that for $1 < p < \infty$ a weight $w$ in $\rn$ (that is, $w \in L^1_{\rm{loc}}(\rn)$ with $w \geq 0$ a.e. in $\rn$) is said to belong to the \emph{Muckenhoupt class} $A_p(\rn)$ if 
\begin{equation}\tag*{$(A_p)$}\label{Ap}
[w]_{A_p}:= \sup\limits_{Q} \left(\fint_Q w(x) \dx \right)  \left(\fint_Q w(x)^{\frac{-1}{(p-1)}} \dx \right)^{p-1} < \infty
\end{equation}
where $Q \subset \rn$ is a cube and $\fint_Q u$ denotes the average $\frac{1}{|Q|} \int_Q u$. The endpoint classes for $p=1$ and $p=\infty$ are defined by
\begin{equation}\tag*{$(A_1)$}\label{def:A1}
[w]_{A_1}:= \sup\limits_{Q} \left(\fint_Q w(x) \dx \right)  \left(\essinf_Q w \right)^{-1} < \infty
\end{equation}
and
\begin{equation}\tag*{$(A_\infty)$}\label{def:Ainf}
[w]_{A_\infty}:= \sup\limits_{Q} \left(\fint_Q w(x) \dx \right)  \exp \left(-\fint_Q \ln w(x) \dx \right) < \infty.
\end{equation}
As it turns out, $
A_\infty(\rn) = \bigcup_{p\geq1}A_p(\rn)
$
(see for instance \cite[Section~9.3]{GrafakosBook}). Typical examples of $A_p$-weights are the locally integrable powers of $|x|$. More precisely, for $1 < p < \infty$, the weight $|x|^a \in A_p(\rn)$ if and only if $-n < a < n (p-1)$ and $|x|^a \in A_1(\rn)$ if and only if $-n < a \leq 0$, see for instance  \cite[p.~286]{GrafakosBook}. When the locally integrable powers of $|x|$ are replaced by those of a distance to a  set $E \subset \rn$, their membership to $\Ainf$, or to a particular $A_p(\rn)$ class, will be determined by corresponding notions of porosity applied to $E$, as described in Section~\ref{secc:distance:weight:porosity}.  In Section~\ref{sec:proofs:dist:E} we will construct examples of $\Ainf$-weights based on the following fact: For $1 < p < \infty$, the class $A_p(\rn)$ forms a lattice in the sense that if $w_1$ and  $w_2$ belong to $A_p(\rn)$, then $\max\{w_1, w_2\}$ and $\min\{w_1, w_2\} $ belong to $A_p(\rn)$ as well, see \cite[Proposition~4.3]{KKM}. Since $\Ainf = \bigcup_{1 \leq p < \infty} A_p(\rn)$ and $A_{p_1}(\rn) \subset A_{p_2}(\rn)$ whenever $1 \leq p_1 \leq p_2 < \infty$, it follows that if $w_1$ and $w_2$ belong to $\Ainf$, then so do $\max\{w_1, w_2\}$ and $\min\{w_1, w_2\}$.

\subsection{Reverse H\"older classes} For $1 < s < \infty$ we write $w \in RH_s(\rn)$ if
\begin{equation*}
[w]_{RH_s}:= \sup\limits_{Q} \left(\fint_Q w(x)^s \dx \right)^{1/s}  \left(\fint_Q w(x) \dx \right)^{-1} < \infty.
\end{equation*}
As it turns out (see for instance \cite[Section~9.3]{GrafakosBook}),
$
\bigcup_{s > 1} RH_s(\rn) = A_\infty(\rn).
$
For $s=\infty$ we write $w \in RH_\infty(\rn)$ if 
\begin{equation*}
[w]_{RH_\infty}:= \sup\limits_{Q} \left(\esssup\limits_{Q} w\right)  \left(\fint_Q w(x) \dx \right)^{-1} < \infty.
\end{equation*}
It is a known fact that the class $RH_\infty(\rn)$ is invariant under positive powers, that is, the implication
\begin{equation}\label{RHinfty:powers}
w \in RH_\infty(\rn), \ell > 0 \Rightarrow w^\ell \in RH_\infty(\rn)
\end{equation}
holds true with $[w^\ell]_{RH_\infty(\rn)}$ depending only on $[w]_{RH_\infty(\rn)}$, $\ell$, and $n$, see \cite[Theorem~4.2]{CUN}.

\begin{remark}\label{comments:A1:RHinfty} If $w \in A_1(\rn)$ then there exists $\theta_0 > 0$ (depending only on $[w]_{A_1(\rn)}$ and $n$) such that $w^{-\theta_0} \in RH_\infty(\rn)$. Conversely, if $w \in RH_\infty(\rn)$, then there exists $\theta_1 >0$ (depending only on $[w]_{RH_\infty(\rn)}$ and $n$) such that $w^{-\theta_1} \in A_1(\rn)$, see \cite[Corollary 4.5]{CUN}. 
\end{remark}

\subsection{Weighted inequalities for the fractional integral operator} Let us record the following two results from \cite{MaSo1} regarding weighted estimates for $I_\alpha$ in the cases $\Theta_{p,q}(\alpha) = 0$ and $\Theta_{p,q}(\alpha) > 0$, with $\Theta_{p,q}(\alpha)$ as in \eqref{def:alpha0}. In the statements below, the notation $(u,v)$ denotes a pair of nonnegative functions $u, v$ defined on $\rn$.

\begin{thm}[Theorem 1.1 from \cite{MaSo1}]\label{thm:SW:Ainf:p:q:alpha0=0} Fix $1 < p \leq q < \infty$ and $0 < \alpha < n$ such that $\Theta_{p,q}(\alpha) = 0$. Given $(u,v)$ with $u, v^{-p'/p} \in \Ainf$ suppose that there exists $C_0 > 0$ with
\begin{equation}\label{u:1/q:C:v:1/p}
u(x)^{1/q} \leq C_0 \, v(x)^{1/p}, \quad \text{a.e. } x \in \rn.
\end{equation}
Then
\begin{equation*}
\left(\int_{\rn} |I_\alpha(f)(x)|^q  u(x) \dx \right)^{1/q}\leq C_1  C_0 \left(\int_{\rn} |f(x)|^p v(x) \dx \right)^{1/p},
\end{equation*}
for every measurable $f$, where $C_1 > 0$ depends only on $n$, $p$, $q$, $\alpha$, $[u]_{\Ainf}$, and $[v^{-p'/p}]_{\Ainf}$.
\end{thm}

For the case $\Theta_{p,q}(\alpha) > 0$ we have

\begin{thm}[Theorem 1.2 from \cite{MaSo1}]\label{thm:SW:Ainf:p:q} Fix $1 < p \leq q < \infty$ and $0 < \alpha < n$ such that $\Theta_{p,q}(\alpha) > 0$ and $(u,v)$ with $u, v^{-p'/p} \in \Ainf$ verifying
\begin{equation}\label{cond:u:v:Theta>0}
\frac{u^{1/q}}{v^{1/p}} \in L^{n/\Theta_{p,q}(\alpha), \infty}(\rn). 
\end{equation}
Then
$$
\left(\int_{\rn} |I_\alpha(f)(x)|^q  u(x) \dx \right)^{1/q}\leq C_2 \norm{u^{1/q}/v^{1/p}}{L^{n/\Theta_{p,q}(\alpha), \infty}(\rn)} \left(\int_{\rn} |f(x)|^p v(x) \dx \right)^{1/p},
$$
for every measurable $f$, where $C_2 > 0$ depends only on $n$, $\alpha$, $p$, $q$, $[u]_{\Ainf}$, and $[v^{-p'/p}]_{\Ainf}$.   
\end{thm}

\begin{remark} The proofs of Theorems \ref{thm:indices:SW}, \ref{thm:gamma:1:2:delta:1:2:p=q},  \ref{thm:indices:SW:q=p*}, and \ref{thm:app:q=p*:E}  will rely on Theorems~\ref{thm:SW:Ainf:p:q:alpha0=0} and \ref{thm:SW:Ainf:p:q} as well as on the pointwise estimate (see \cite[Lemma 7.14]{gt})
\begin{equation}\label{f:I1:grad:f}
|f(x)| \leq \frac{1}{n\omega_n} I_1(|\nabla f| )(x), \quad \forall f \in C_0^1(\rn), x \in \rn,
\end{equation}
where $\omega_n$ denotes the Euclidean measure of a unit ball in $\rn$.

Although we focus on Hardy-Sobolev inequalities, weighted Poincar\'e-type inequalities can be obtained as well by using that for a convex set $S \subset \rn$ and $f \in C^1(S)$ we have
\begin{equation}\label{RF:P}
\bigg|f(x) - \fint_S f \bigg| \leq \frac{\diam{S}^n}{n |S|} I_1(|\nabla f| \chi_S )(x), \quad \forall x \in S,
\end{equation}
see for instance \cite[Lemma 7.16]{gt}. From \eqref{RF:P} and Theorems \ref{thm:SW:Ainf:p:q:alpha0=0} and \ref{thm:SW:Ainf:p:q} (always in the case $\alpha=1$), given a cube $Q \subset \rn$ the weighted inequality
\begin{equation*}
\left(\int_{Q} \bigg|f(x) - \fint_Q f \bigg|^q u(x) \, dx \right)^{1/q} \leq C \left(\int_{Q} |\nabla f(x)|^p v(x) \, dx \right)^{1/p}
\end{equation*}
holds true for every $f \in C^1(\rn)$ whenever a pair $(u,v)$, with $u, v^{-p'/p} \in \Ainf$, satisfies \eqref{u:1/q:C:v:1/p} or  \eqref{cond:u:v:Theta>0}. Moreover, for $0 < \alpha < n$, the pointwise inequality \eqref{f:I1:grad:f} can be replaced with the fractional-derivative identity
$$
f = I_\alpha((-\Delta)^{\alpha/2}(f))
$$
for every $f$ in the Schwartz class $\mathcal{S}(\rn)$ (see for instance \cite[p.117]{Stein}) and then Theorems \ref{thm:SW:Ainf:p:q:alpha0=0} and \ref{thm:SW:Ainf:p:q} yield
\begin{equation*}
\left(\int_{\rn} |f(x)|^q  u(x) \dx \right)^{1/q}\leq C  \left(\int_{\rn} |(-\Delta)^{\alpha/2}f(x)|^p v(x) \dx \right)^{1/p}, \quad \forall f \in \mathcal{S}(\rn),
\end{equation*}
 whenever a pair $(u,v)$, with $u, v^{-p'/p} \in \Ainf$, satisfies \eqref{u:1/q:C:v:1/p} or  \eqref{cond:u:v:Theta>0}.
\end{remark}

\subsection{Dimensions of sets} For $\delta > 0$, $E_\delta$ stands for the \emph{$\delta$-neighborhood of $E$}, that is,
\begin{equation*}
E_\delta := \{x \in \rn: \dist{x, E} < \delta\}.
\end{equation*}
Notice that, given $\delta > 0$, we have $|E_\delta| <\infty$ if and only if $E$ is bounded. 

\subsubsection{The upper Minkowski dimension of $E \subset \rn$}\label{secc:dim:M} Given a set $E \subset \rn$ and  $\delta > 0$, let $N(E, \delta)$ denote the smallest number of closed balls of radius $\delta$ required to cover $E$. In particular, given $\delta > 0$,  $N(E, \delta) < \infty$ if and only if  $E$ is bounded. The \emph{upper box-counting} or \emph{upper Minkowski} dimension of $E$ is defined as
\begin{equation}\label{def:dimM}
\overline{\rm{dim}_M}(E):=\limsup_{\delta \to 0^+} \frac{\log N(E, \delta)}{\log (1/\delta)},
\end{equation}
see for instance \cite[p.~43]{Falconer}. Equivalently, $\overline{\rm{dim}_M}(E)$ can be regarded as the infimum of the nonnegative $\lambda$'s such that there exists $C >0$ (which may depend on $\lambda$ and $E$) satisfying
\begin{equation}\label{E:delta:lambda}
|E_\delta| \leq  C \delta^{n-\lambda}, \quad \forall 0 < \delta < \diam{E},
\end{equation}
see for instance \cite[Remark 6.8]{ALMV}.  Notice that $\overline{\rm{dim}_M}(E) < \infty$ if and only if $E$ is bounded. In addition, we have $\overline{\rm{dim}_M}(E) = \overline{\rm{dim}_M}(\overline{E})$ for every $E \subset \rn$ by \cite[Proposition 3.4]{Falconer}. Let us record here a simple result concerning  $\overline{\rm{dim}_M}(E)$ that will be useful later on.

\begin{lem}\label{lemma:s>dimB(F)} Let  $\emptyset \neq E \subset \rn$ be a  bounded set and fix $0 < s < n$. If $\overline{\rm{dim}_M}(E) < s$, then there exists $C > 0$ such that
\begin{equation}\label{meas:F:delta}
|E_\delta| \leq C \delta^{n-s}, \quad \forall 0 < \delta < \diam{E}.
\end{equation}
\end{lem}

\begin{proof} Set $d_B:= \overline{\rm{dim}_M}(E)$ and $f(\delta):= \frac{\log N(E, \delta)}{\log (1/\delta)}$, then \eqref{def:dimM} means
$d_B = \inf\limits_{a>0} \sup f((0, a)).$ Put $\eps_0:= s - \overline{\rm{dim}_M}(E) >0$, then there exists $a_0 > 0$ (depending on $\eps_0$ and $E$) such that
$$
\sup f((0,a_0)) < d_B + \eps_0,
$$
that is, 
\begin{equation}\label{bound:f:delta}
f(\delta) < d_B + \eps_0, \quad \forall \delta \in (0, a_0).
\end{equation}
Now, given $\delta \in (0, \diam{E})$, let us cover $E$ with $N(E, \delta) < \infty$ closed balls of radius $\delta$, in particular, the set $\{x \in \rn: \dist{x, E} < \delta\}$ can be covered with $N(E, \delta)$ closed balls concentric with the first balls but with radius $2 \delta$. Hence,
$
|E_\delta| \leq \omega_n N(E, \delta)  (2\delta)^n,
$
with $\omega_n :=|B(0,1)|$. Now, if $\delta \in (0, a_0)$, then \eqref{bound:f:delta} implies
$$
\frac{\log N(E, \delta)}{\log (1/\delta)} < d_B + \eps_0 = s,
$$
which means
$
N(E, \delta) < \delta^{-s}
$
and then
$
|E_\delta| \leq \omega_n 2^n \delta^{n-s}.
$
If $\delta \geq a_0$, fix $x_0 \in E$ and use the inclusion
$
E_\delta \subset B(x_0, 2 \diam{E}),
$
so that
$$
|E_\delta | \leq \omega_n 2^n \diam{E}^n \leq  \omega_n 2^n \diam{E}^n\left(\frac{\delta}{a_0}\right)^{n-s}
$$
and \eqref{meas:F:delta} follows with $C = \omega_n 2^n  \max\{1,  \diam{E}^n/a_0^{n-s} \}$.  
\end{proof}

\subsubsection{The Assouad dimension of $E \subset \rn$}\label{sec:Assouad} The \emph{Assouad dimension of a set $E \subset \rn$}, denoted by ${\rm{dim}}_A(E)$, is defined as the infimum of $\lambda \geq 0$ such that there exists $C > 0$ with 
\begin{equation}\label{def:dimA:lambda}
N(E \cap B(x,R), r) \leq C \left(\frac{R}{r} \right)^\lambda, \quad \forall x \in E, 0 < r < R.
\end{equation}
Notice that if $0 < \diam{E} < \infty$ and if $\lambda \geq 0$ satisfies \eqref{def:dimA:lambda}, then by fixing an $x \in E$ and $R:=\diam{E}$ we get $E \cap B(x,R) = F$ as well as
$$
N(E, r) \leq C \diam{E}^{\lambda} r^{-\lambda}, \quad \forall 0 < r <  \diam{E},
$$
which together with \eqref{def:dimM} implies $\overline{{\rm{dim}}}_M(E) \leq \lambda$ and thus 
\begin{equation}\label{dimM<dimA}
\overline{{\rm{dim}}}_M(E) \leq {\rm{dim}}_A(E)
\end{equation}
whenever $E \subset \rn$ is bounded. For every $E \subset \rn$ we have ${\rm{dim}}_A(E) \leq n$ (see for instance \cite[Remark 10.18]{KLV}) and the \emph{Assouad co-dimension of $E$} is defined as $\underline{{\rm{codim}}}_A(E):= n -{\rm{dim}}_A(E)$. Equivalently, (see \cite[Lemma~3.4]{KLVu}) $\underline{{\rm{codim}}}_A(E)$ equals the supremum of $\nu \geq 0$ such that there exists $C > 0$ with
\begin{equation*}
\frac{|E_r \cap B(x,R)|}{|B(x,R)|} \leq C \left(\frac{R}{r} \right)^{- \nu},  \quad \forall x \in E, 0 < r < R.
\end{equation*}

\subsection{Distance weights and notions of porosity}\label{secc:distance:weight:porosity}

\cite{ALMV, IGV26, KLV, Pa-UT}

\subsubsection{Porous sets}\label{secc:porous:sets} Following \cite[Definition 10.11]{KLV}, a set $E \subset \rn$ is \emph{porous} if there exists $C > 0$ such that for every $x \in \rn$ and $r > 0$ there exists $y \in \rn$ satisfying $B(y, Cr) \subset B(x, r) \setminus E$. The following are equivalent (see for instance \cite[Theorem~10.25]{KLV}): 
\begin{enumerate}[(i)]
\item a nonempty set $E \subset \rn$ is porous,
\item ${\rm{dim}}_A(E) < n$. 
\end{enumerate}

\subsubsection{Weakly porous sets}\label{secc:WPS} In \cite{ALMV},  Anderson, Lehrb\"ack,  Mudarra, and V\"ah\"akangas introduced and developed the geometric concept of weak porosity for a nonempty set $E \subset \rn$. In terms of the weight $\dist{\cdot, E}$,  by \cite[Theorem~1.1]{ALMV} the following are equivalent:
\begin{enumerate}[(i)]
\item a nonempty set $E \subset \rn$ is weakly porous,
\item there exists $\alpha > 0$ such that $\dist{\cdot, E}^{-\alpha} \in A_1(\rn)$.
\end{enumerate}
The precise range of $\alpha$'s that make $\dist{\cdot, E}^{-\alpha} \in A_1(\rn)$ or $\dist{\cdot, E}^{-\alpha} \in A_p(\rn)$ for $1 < p < \infty$ for a weakly porous set $E \subset \rn$ relies on the Muckenhoupt exponent ${\rm{Mu}}(E)$ defined in  \cite[Definition 6.1]{ALMV} and which we are denoting as ${\rm{Mu}}_1(E)$ for the sake of consistency. As mentioned, ${\rm{Mu}}_1(E)$ refines $\underline{{\rm{codim}}}_A(E)$ in the sense that $E$ weakly porous if and only if ${\rm{Mu}}_1 (E) >0$ with ${\rm{Mu}}_1 (E) = \underline{{\rm{codim}}}_A(E)$ whenever $E$ is porous. Then, by \cite[Theorem~1.2]{ALMV}, given a weakly porous set $E \subset \rn$ we have
\begin{enumerate}[(i)]
\item $\dist{\cdot, E}^{-\alpha} \in A_1(\rn)$ if and only if $0 \leq \alpha <  {\rm{Mu}}_1(E)$, and
\item for $1 < p < \infty$, $\dist{\cdot, E}^{-\alpha} \in A_p(\rn)$ if and only if $(1-p) {\rm{Mu}}_1(E) < \alpha <  {\rm{Mu}}_1(E)$.
\end{enumerate}

\begin{remark}\label{rmk:v:beta:WP} Notice that given a weakly porous set $E \subset \rn$, $1 < p < \infty$, and $\beta > 0$ with $\beta p'/p < {\rm{Mu}}_1(E)$, by setting $v(x):= \dist{x, E}^{\beta}$ we have $v^{-p'/p} =\dist{\cdot, E}^{-\beta p'/p} \in A_1(\rn)$ and consequently $v \in RH_\infty(\rn)$ due to \eqref{RHinfty:powers} and Remark \ref{comments:A1:RHinfty}. 
\end{remark}

\subsubsection{Median porous sets} Recently, in \cite{Pa-UT}, M.~Pasquariello and I.~Uriarte--Tuero and, independently, I.~G\'omez-Vargas \cite{IGV26} extended the notion of weak porosity, by introducing Muckenhoupt exponents $\Mu_p(E)$ for $1 < p \leq \infty$ for a nonempty set $E \subset \rn$. In particular, $E$ is called median porous  if and only if ${\rm{Mu}}_\infty (E) >0$ where ${\rm{Mu}}_p (E) = \underline{{\rm{codim}}}_A(E)$ for every $1 \leq p \leq \infty$ whenever $E$ is weakly porous. Now, by \cite[Theorem~8.7 and Corollary 8.10]{Pa-UT}, given a nonempty set $E \subset \rn$ we have
$$
\dist{\cdot, E}^{-\beta} \in \Ainf \Leftrightarrow - \sup\limits_{p>1}(p-1) \Mu_{p'}(E) < \beta < \Muinf(E)
$$
and the following are equivalent:
\begin{enumerate}[(i)]
\item $E$ is median porous,
\item there exists $\beta > 0$ such that $\dist{\cdot, E}^{-\beta} \in \Ainf$,
\item $\Muinf(E) >0$,
\item $\Mu_p >0$ for some $1 < p < \infty$,
\item $\Mu_p > 0$ for every $1 < p < \infty$.
\end{enumerate}
Let us set
$\ell_\infty(E):= - \sup\limits_{p>1}(p-1) \Mu_{p'}(E) 
$ and
\begin{equation}\label{def:I:infty}
I_\infty(E):= (\ell_\infty(E),  \Muinf(E)). 
\end{equation}
That is, $\dist{\cdot, E}^{-\beta} \in \Ainf$ if and only if $\beta \in I_\infty(E).$ Also, by \cite[Corollary 8.10.3]{Pa-UT} it follows that $\ell_\infty(E) < 0$ whenever $\Muinf(E) >0$. Also, by \cite[Corollary 8.10]{Pa-UT}, if $E \subset \rn$ then $\Muinf(E) \leq n$ and $\Muinf(E) := \sup\limits_{1 < q < \infty} \Mu_q(E).$ Moreover, by \cite[Corollary 8.10]{Pa-UT}, if $E$ is weakly porous, then it is median porous with $\Muinf(E)=\Mu_p(E) > 0$ for every $1 \leq p < \infty$.  The next lemma guarantees that $\overline{{\rm{dim}}}_M(E) < n$ for every bounded, median porous set  $E \subset \rn$.

\begin{lem}\label{lemma:Muinf<n-dimM} Given a nonempty bounded set $E \subset \rn$, we have
\begin{equation}\label{Muinf<n-dimM}
 \Muinf(E) \leq n - \overline{{\rm{dim}}}_M(E). 
\end{equation}
\end{lem}

\begin{proof} If $\Muinf(E) =0$ then there is nothing to prove due to \eqref{dimM<dimA} and the aforementioned fact that ${\rm{dim}}_A(E) \leq n$ for every $E \subset \rn$. If $\Muinf(E)  > 0$, let us take $\alpha \in (0, \Muinf(E))$ which makes for $\dist{\cdot, E}^{-\alpha} \in \Ainf \subset L^1_{{\rm{loc}}}(\rn)$. Next, by \cite[Lemma~3]{Zu}, given a set $A \subset \rn$, an open set $\Omega \subset \rn$,  and $\gamma >0$ we have
\begin{equation}\label{lemma:Zu}
\int_{A_r \cap \Omega} \dist{x, A}^{-\gamma} dx = r^{-\gamma} |A_r \cap \Omega| + \gamma \int_0^r |A_t \cap \Omega| t^{-\gamma -1} dt, \quad \forall r > 0. 
\end{equation}
Now, for a bounded set $E \subset \rn$ and $0 < r < 1$ let us choose $A:=E$, $\Omega := E_1$ so that $E_r \cap \Omega = E_r$, and $\gamma := \alpha$ so that $\dist{\cdot, E}^{-\alpha} \in L^1_{{\rm{loc}}}(\rn)$. The identity \eqref{lemma:Zu} then yields
$$
r^{-\alpha} |E_r| \leq \int_{E_1} \dist{x, A}^{-\alpha} dx =: \eta < \infty.
$$
Define $\lambda:= n -\alpha$ (which is positive since $\alpha < \Muinf(E) \leq n$) so that $|E_r| \leq \eta r^{n-\lambda}$ for every $0 < r < 1$, and from the definition of $\overline{{\rm{dim}}}_M(E)$ in \eqref{E:delta:lambda} it follows that $\overline{{\rm{dim}}}_M(E) \leq \lambda$, that is, $\alpha \leq n - \overline{{\rm{dim}}}_M(E)$. Finally, letting $\alpha \to \Muinf(E)$ proves \eqref{Muinf<n-dimM}. 
\end{proof}

\section{On a function $g_E \in L^{1, \infty}(\rn)$}\label{sec:gE} 

In this section we introduce and study a function $g_E$ associated to a bounded set $E \subset \rn$ and characterize the condition $g_E \in L^{1, \infty}(\rn)$ in terms of measure-theoretic properties of $E$, this characterization plays a central role in the construction of examples from Section~\ref{sec:proofs:dist:E}. 

\begin{lem}\label{lemma:d:x:F:L1w} Fix $0 \leq s <n$ and a set $E \subset \rn$ with $0< \diam{E} < \infty$ and $|\overline{E}|=0$. For a.e. $x \in \rn$ define
$$
g_E(x) :=   \left\{
      \begin{array}{cl}
       \left(\frac{\dist{x, E}}{\diam{E}}\right)^{-(n-s)}  &  \text{if } \dist{x, E} < \diam{E} \\
	\left(\frac{\dist{x, E}}{\diam{E}}\right)^{-n}  &  \text{if } \dist{x, E} \geq \diam{E}.    
\end{array}
    \right. 
$$ 
Then, if there exists $K > 0$ such that 
\begin{equation}\label{fat:r:n:s:K}
|E_r| \leq K r^{n-s}, \quad \forall 0 < r < \diam{E},
\end{equation}
it follows that $g_E \in L^{1,\infty}(\rn)$ with
\begin{equation}\label{norm:gF:L1w}
\norm{g_E}{L^{1,\infty}(\rn)} \leq K \diam{E}^{n-s}+ \omega_n 2^n \diam{E}^n.
\end{equation}
Conversely, if $g_E \in L^{1,\infty}(\rn)$, then \eqref{fat:r:n:s:K} holds with $K:=\norm{g_E}{L^{1,\infty}(\rn)}/\diam{E}^{n-s}.$
\end{lem}

\begin{proof} For $t > 0$ let us write
\begin{align*}
|\{x \in \rn:  g_E(x) > t\}| &  = |\{x \in \rn:  \dist{x, E} < \diam{E} \text{ and } g_E(x) > t\}|\\
& \quad + |\{x \in \rn:  \dist{x, E} \geq \diam{E} \text{ and } g_E(x) > t\}|  =: I + II. 
\end{align*}
By setting $r:= \diam{E} \min \{t^{-\frac{1}{n-s}}, 1\}$ and using it in \eqref{fat:r:n:s:K} we get
$$
I = |\{x \in \rn:  \dist{x, E} < r\} | \leq K \diam{E}^{n-s} t^{-1}.
$$
On the other hand, by defining $R:= \diam{E}/t^{1/n}$, we can write 
$$
II = |\{x \in \rn:  \diam{E} \leq \dist{x, E} < R\}|
$$
and notice that, for any fixed $x_E \in E$, the inclusion
\begin{equation}\label{incl:<R:B2R}
\{x \in \rn:  \diam{E} \leq \dist{x, E} < R\} \subset \overline{B(x_E, 2R)}
\end{equation}
holds true, since given $x \in \rn$ with $\diam{E} \leq \dist{x, E} < R$ and $\eps > 0$ let $y_E \in E$ satisfy $|x-y_E| \leq \dist{x, E} + \eps$ and then
$$
|x - x_E| \leq |x - y_E| + |y_E - x_E| \leq \dist{x, E} + \diam{E} +\eps < R + \diam{E} +\eps < 2 R +\eps,
$$ 
and we obtain \eqref{incl:<R:B2R} by letting $\eps \to 0$. Consequently, 
$$
II \leq |B(0,2)| R^n = |B(0,2)| \diam{E}^n/t
$$
and \eqref{norm:gF:L1w} follows. Conversely,   $g_E \in L^{1,\infty}(\rn)$ means
\begin{equation}\label{hyp:gF:L1w}
t |\{x \in \rn:  g_E(x) > t\}| \leq \norm{g_E}{L^{1,\infty}(\rn)}, \quad \forall t >0
\end{equation}
and, given $0 < r \leq \diam{E}$, for $x \in \rn$ with $\dist{x, E} < r$ we finally get $\dist{x, E} < \diam{E}$ and then $g_E(x) =  \left(\frac{\dist{x, E}}{\diam{E}}\right)^{-(n-s)}$. By using \eqref{hyp:gF:L1w} with $t:= (r/\diam{E})^{-(n-s)}$,
\begin{align*}
|\{x \in \rn: \dist{x, E} < r\}| &  \leq |\{x \in \rn: \diam{E} g_E(x)^\frac{1}{-(n-s)} < r\}| \\
& = |\{x \in \rn: g_E(x) > (r/\diam{E})^{-(n-s)}\}\\
& \leq \norm{g_E}{L^{1,\infty}(\rn)} \left(\frac{r}{\diam{E}} \right)^{n-s}
\end{align*}
which proves \eqref{fat:r:n:s:K} with $K:= \norm{g_E}{L^{1,\infty}(\rn)}/\diam{E}^{n-s}$. 
\end{proof}

\section{Proofs of Theorems \ref{thm:indices:SW}, \ref{thm:gamma:1:2:delta:1:2:p=q},  \ref{thm:indices:SW:q=p*}, \ref{thm:app:q=p*:E} and Corollaries \ref{coro:v=1:p<q<p*}, \ref{coro:beta:1:2:critical}}\label{sec:proofs:dist:E}

The proofs are based on the function $g_E$ from Section~\ref{sec:gE} and on the right choice of indices $\gamma_1, \gamma_2,\delta_1, \delta_2 \in \re$ so that the pair $(u,v):=(u_{\gamma_1, \gamma_2},u_{\delta_1, \delta_2})$ satisfies either the condition \eqref{cond:u:v:Theta>0} from Theorem \ref{thm:SW:Ainf:p:q} (when $\Theta_{p,q}(\alpha) > 0$)  or \eqref{u:1/q:C:v:1/p} from Theorem \ref{thm:SW:Ainf:p:q:alpha0=0} (when $\Theta_{p,q}(\alpha) = 0$), whenever $E \subset \rn$ is a bounded, median porous set. Recall that $u_{\gamma_1, \gamma_2}$ is defined for $x \in \rn \setminus E$ as in \eqref{def:u:gamma:1:2:c} and that it can be rewritten as \eqref{def:u:F}.

\medskip

\emph{Proof of Theorem \ref{thm:indices:SW}.} Fix $1 < p \leq q < \infty$ and $0 < \alpha < n$ such that $\Theta_{p,q}(\alpha) > 0$ (always with $\Theta_{p,q}(\alpha)$ as in  \eqref{def:alpha0}). Let $E \subset \rn$ be a median porous set with $0 < \diam{E} < \infty$. In particular, $\overline{\rm{dim}_M}(E) <~n$ (due to \eqref{Muinf<n-dimM}) and $\dist{\cdot, E}^{-\vartheta} \in \Ainf$ if and only if $\vartheta  \in I_\infty(E)$ as defined in \eqref{def:I:infty}. Next consider exponents $\gamma_1, \gamma_2$ as in \eqref{cond:gamma:1:2}. Notice that the local integrability of $\dist{\cdot, E}^{-\vartheta}$ for $\vartheta \in I_\infty(E)$ along with the fact that $\dist{\cdot, E} = \dist{\cdot, \overline{E}}$ implies $|\overline{E} |=0$ and then $u_{\gamma_1, \gamma_2}(x)$ is well defined for a.e. $x \in \rn$.  Now, the hypothesis \eqref{cond:gamma:1:2} means precisely $\dist{\cdot, E}^{\gamma_1}, \dist{\cdot, E}^{\gamma_2} \in \Ainf$ which together with \eqref{def:u:F} and the fact (mentioned in Section~\ref{sec:Ap:weights}) that $\Ainf$ forms a lattice yields $u_{\gamma_1, \gamma_2} \in~\Ainf$. Similarly,  $\delta_1, \delta_2$ as in \eqref{cond:delta:1:2} implies $u_{\delta_1, \delta_2}^{-p'/p} \in \Ainf$. Next, we use \eqref{def:u:F}  to write
$$
\left(\frac{u_{\gamma_1, \gamma_2}(x)^\frac{1}{q}}{u_{\delta_1, \delta_2}(x)^\frac{1}{p}}\right)^\frac{n}{\Theta_{p,q}(\alpha)} = \left\{
      \begin{array}{cl}
       \left(\frac{\dist{x, E}}{\diam{E}}\right)^{\frac{n}{\Theta_{p,q}(\alpha)}(\frac{\gamma_1}{q} - \frac{\delta_1}{p})}  &  \text{if } \dist{x, E} < \diam{E} \\
	\left(\frac{\dist{x, E}}{\diam{E}}\right)^{\frac{n}{\Theta_{p,q}(\alpha)}(\frac{\gamma_2}{q} - \frac{\delta_2}{p})}  &  \text{if } \dist{x, E} \geq \diam{E}, 
\end{array}
    \right. 
$$
and by choosing $\gamma_1, \gamma_2, \delta_1, \delta_2$ such that
\begin{equation}\label{cond:gammas:deltas}
\left\{
      \begin{array}{rl}
       \frac{n}{\Theta_{p,q}(\alpha)}\left(\frac{\gamma_1}{q} - \frac{\delta_1}{p}\right) &= - (n-s) \\
	\frac{n}{\Theta_{p,q}(\alpha)}\left(\frac{\gamma_2}{q} - \frac{\delta_2}{p}\right) & = -n, 
\end{array}
    \right. 
\end{equation}
for some $s \in (\overline{\rm{dim}_M}(E), n)$ it follows that $\left(\frac{u_{\gamma_1, \gamma_2}^{1/q}}{u_{\delta_1, \delta_2}^{1/p}}\right)^{n/\Theta_{p,q}(\alpha)} = g_E$, with $g_E$ as in Lemma~\ref{lemma:d:x:F:L1w} satisfying 
$$
\norm{g_E}{L^{1, \infty}(\rn)} \leq K \diam{E}^{n-s}+ \omega_n 2^n \diam{E}^n
$$ 
for some $K > 0$ depending on $s$, $n$ and $E$, due to Lemma~\ref{lemma:s>dimB(F)} (since $\overline{\rm{dim}_M}(E) < s < n$). Consequently, 
$$
 \norm{u^{1/q}/v^{1/p}}{L^{n/\Theta_{p,q}(\alpha), \infty}(\rn)}  =  \norm{g_E}{L^{1, \infty}(\rn)}^{\Theta_{p,q}(\alpha)/n}  \leq (K \diam{E}^{n-s}+ \omega_n 2^n \diam{E}^n)^{\Theta_{p,q}(\alpha)/n}. 
$$
Finally, by solving for $s$ from the first equality in \eqref{cond:gammas:deltas}, the inequalities $\overline{\rm{dim}_M}(E) < s < n$ amount to \eqref{cond:s:dimB:n}, while the second equality in \eqref{cond:gammas:deltas} means \eqref{cond:gamma2:delta2}. Thus the conditions $u_{\gamma_1, \gamma_2}, u_{\delta_1, \delta_2}^{-p'/p} \in \Ainf$ and $u_{\gamma_1, \gamma_2}^{1/q}/u_{\delta_1, \delta_2}^{1/p} \in L^{n/\Theta_{p,q}(\alpha), \infty}(\rn)$ in Theorem~\ref{thm:SW:Ainf:p:q} are met and \eqref{bound:T:alpha:SW} follows from it. \qed

\medskip

\emph{Proof of Corollary \ref{coro:v=1:p<q<p*}.} Since $\delta_1=\delta_2 = 0$, the condition $\delta_j p'/p \in I_\infty(E)$ holds trivially for $j=1,2$ (since $ \ell_\infty(E) < 0 < \Muinf(E)$), also $u_{\delta_1, \delta_2} \equiv 1$. Then \eqref{cond:gamma:1:2}, \eqref{cond:delta:1:2},  \eqref{cond:s:dimB:n}, and \eqref{cond:gamma2:delta2} reduce to
$$
\left\{
      \begin{array}{rl}
     -\gamma_j  &\in I_\infty(E) \quad \text{for } j=1,2\\[2 mm]
      \overline{\rm{dim}_M}(E) - n & < \frac{n}{q \Theta_{p,q}(1)} \gamma_1 <0  \\[2 mm]
	\gamma_2 & = - q \Theta_{p,q}(1)= \frac{q}{p}(n-p) -n <0,
\end{array}
    \right. 
$$
which are equivalent to \eqref{cond:gamma:1:2:delta=0:a} and \eqref{cond:gamma:1:2:delta=0:b}. Notice that the second inequality for $\gamma_2$ in \eqref{cond:gamma:1:2:delta=0:b}, along with the fact that $\Muinf(E) \leq n$ by \cite[Corollary 8.10.6]{Pa-UT}, imposes $p < n$, while the first one makes $q < np/(n-p)$. Finally, since $\gamma_1 > \gamma_2 \left(1- \frac{\overline{\rm{dim}_M}(E)}{n} \right)$, with $\gamma_2 < 0$, it follows that $\gamma_1 > \gamma_2$ and \eqref{def:u:gamma:1:2:c} yields
$$
u_{\gamma_1, \gamma_2}(x) =\min \left\{ \left(\frac{\dist{x, E}}{\diam{E}}\right)^{\gamma_1}, \left(\frac{\dist{x, E}}{\diam{E}}\right)^{\gamma_2} \right\}.  
$$
Finally, the weighted Hardy-Sobolev inequality \eqref{Hardy:p:q:E:q<p*} follows from \eqref{I1:v=1:p<q<p*}  and \eqref{f:I1:grad:f}.  \qed

\medskip

\emph{Proof of Theorem \ref{thm:gamma:1:2:delta:1:2:p=q}.} In this case $\alpha=\Theta_{p,q}(\alpha) =1$ (in particular, $n > 1$) and \eqref{cond:gamma:1:2}, \eqref{cond:delta:1:2},  \eqref{cond:s:dimB:n}, and \eqref{cond:gamma2:delta2} reduce to \eqref{cond:gamma:1:2:delta:1:2:p=q} and therefore \eqref{gamma:1:2:delta:1:2:p=q} follows from Theorem \ref{thm:indices:SW}. In turn, \eqref{gamma:1:2:delta:1:2:p=q}  and \eqref{f:I1:grad:f} yield \eqref{HS:Theta1=1}
\qed

\medskip

\emph{Proof of Theorem \ref{thm:indices:SW:q=p*}.} Having $\Theta_{p,q}(\alpha) =0$ means $q = np/(n-\alpha p)$. In this critical case our guide will be Theorem~\ref{thm:SW:Ainf:p:q:alpha0=0}. We will choose the indices $\gamma_1, \gamma_2, \delta_1, \delta_2$ such that $u^{1/q} \leq v^{1/p}$, always with the pair $(u,v):=(u_{\gamma_1, \gamma_2},u_{\delta_1, \delta_2})$. Indeed, by considering the cases $\dist{x, E} < \diam{E}$ and $\dist{x, E} \geq \diam{E}$ the inequality $u^{1/q} \leq v^{1/p}$ is achieved by choosing $\gamma_1, \gamma_2, \delta_1, \delta_2$ with $\gamma_1/q - \delta_1/p \geq 0$ and $\gamma_2/q - \delta_2/p \leq 0$, which amount to \eqref{gamma1:delta1:q=p*} and  \eqref{gamma2:delta2:q=p*}, and together with \eqref{cond:gamma:1:2} and \eqref{cond:delta:1:2} (always to guarantee  $u_{\gamma_1, \gamma_2}, u_{\delta_1, \delta_2}^{-p'/p} \in \Ainf$) as well as Theorem~\ref{thm:SW:Ainf:p:q:alpha0=0}, yield \eqref{bound:T:alpha:SW:2}.\qed

 \medskip
 
\emph{Proof of Corollary \ref{coro:beta:1:2:critical}.} Given $\beta_1, \beta_2$ verifying \eqref{coro:beta:1:2:a}, define $\gamma_1:= \beta_2$, $\gamma_2:=\beta_2$, $\delta_1:= \beta_1$ and $\delta_2:=\beta_1$. Then, due to \eqref{coro:beta:1:2:a}, the exponents  $\gamma_1, \gamma_2, \delta_1, \delta_2 \in \re$ satisfy \eqref{gamma1:delta1:q=p*} and \eqref{gamma2:delta2:q=p*}. Also, by \eqref{coro:beta:1:2:a} since $\gamma_1= \gamma_2= \beta_2 $ we get $-\gamma_1= -\gamma_2 \in  I_\infty(E)$ and the condition \eqref{cond:gamma:1:2} is met as is the condition \eqref{cond:delta:1:2} since $p' \beta_1 / p \in I_\infty(E)$ means precisely \eqref{coro:beta:2:a}. 
\qed

\medskip
 
\emph{Proof of Theorem \ref{thm:app:q=p*:E}} It follows from Theorem~\ref{thm:indices:SW:q=p*} applied with $\alpha =1$ and by noticing that the  conditions on $\gamma_1, \gamma_2, \delta_1, \delta_2$ coincide with \eqref{cond:gamma:1:2}, \eqref{cond:delta:1:2}, \eqref{gamma1:delta1:q=p*} and \eqref{gamma2:delta2:q=p*} when $\alpha =1$ and $q=np/(n-p)$. \qed

\section{Proof of Theorem \ref{thm:Hardy:ineq:p:q:WP:2}}\label{sec:proof:WP:2} We will make use of the following result from \cite{MaSo1}. 

\begin{thm}[Theorem 6.8 from \cite{MaSo1}]\label{thm:Mg:p<q:app} Fix $n > 1$ and $1 < p  \leq q < n/\Theta_{p,q}(1)$ with $\Theta_{p,q}(1) :=1 + n (1/q-1/p) > 0$. Given $v \in RH_\infty(\rn)$ such that $v^{-p'/p} \in \Ainf$  there exists $ C >0$ depending only on $n$, $p$, $q$, $[v^{-p/p}]_{\Ainf}$, and $[v]_{RH_\infty(\rn)}$ such that
$$
\left(\int_{\rn} |f(x)|^q  \mathcal{M}(g)(x)^{\Theta_{p,q}(1)q/n} v(x)^{q/p} \dx \right)^{1/q}\leq C \norm{g}{L^1(\rn)}^{\Theta_{p,q}(1)/n} \left(\int_{\rn} |\nabla f(x)|^p v(x) \dx \right)^{1/p}
$$
for every $f \in C_0^1(\rn)$ and $g \in L^1(\rn)$. 
\end{thm}

In the case where $v$ is a power of $\dist{\cdot, E}$ for a weakly porous set $E \subset \rn$, we obtain the next corollary of Theorem \ref{thm:Mg:p<q:app}.

\begin{coro}\label{coro:Hardy:ineq:p:q:WP}  Fix $n > 1$ and $1 < p  \leq q < n/\Theta_{p,q}(1)$ with $\Theta_{p,q}(1) :=1 + n (1/q-1/p) > 0$. Given a nonempty weakly porous set $E \subset \rn$ and $\beta > 0$ with $\beta p'/p < {\rm{Mu}}_1(E)$, there exists $C > 0$, depending only on $\beta$, $p$, $q$, ${\rm{Mu}}_1(E)$, and $n$, such that
\begin{equation*}
\left(\int_{\rn} |f(x)|^q  \mathcal{M}(g)(x)^{\Theta_{p,q}(1)q/n} d_E(x)^{q \beta/p} \dx \right)^\frac{1}{q}\leq C \norm{g}{L^1(\rn)}^{\Theta_{p,q}(1)/n} \left(\int_{\rn} |\nabla f(x)|^p d_E(x)^{\beta} \dx \right)^\frac{1}{p}
\end{equation*}
for every $f \in C_0^1(\rn)$ and $g \in L^1(\rn)$. 
\end{coro}

\begin{proof} A consequence of Theorem \ref{thm:Mg:p<q:app} by setting $v(x):= \dist{x, E}^{\beta}$ so that $v^{-p'/p} \in A_1(\rn)$ due to $0 < \beta p'/p < {\rm{Mu}}_1(E)$ and then $v \in RH_\infty(\rn)$ by Remark \ref{rmk:v:beta:WP}.
\end{proof}

The proof of Theorem \ref{thm:Hardy:ineq:p:q:WP:2} will also rely on the following lemma.  

\begin{lem}\label{lemma:M:d:gamma:dk} Fix a set $E \subset \rn$  with $0 < \diam{E} < \infty$, and for $\kappa > 0 $ let us write $d_{\kappa}:= \kappa \,  \diam{E}$. Then, for each $\gamma < n - \overline{{\rm{dim}}}_M(E)$ there exists $K > 0$ depending on $\gamma$, $\kappa$, $\diam{E}$, and $n$ such that
\begin{equation}\label{M:d:gamma:dk}
\mathcal{M}(\dist{\cdot, E}^{-\gamma} \chi_{E_{d_{\kappa}}}) (x) \geq   \left\{
      \begin{array}{cl}
       \dist{x, E}^{-\gamma}  &  \text{if } \dist{x, E} < d_{\kappa} \\
	K \, \dist{x, E}^{-n} &  \text{if } \dist{x, E} \geq d_{\kappa}, 
\end{array} \quad \text{for a.e. } x \in \rn. 
    \right. 
\end{equation}
\end{lem}

\begin{proof} Let us first notice that the condition $\gamma < n - \overline{{\rm{dim}}}_M(E)$ guarantees
\begin{equation}\label{dist:gamma:R:L1}
\int_{E_R} \dist{z, E}^{-\gamma} dz < \infty
\end{equation}
for every $R > 0$, see for instance \cite[Lemma~1]{Zu}. Now, for a.e. $x$ with $\dist{x, E} < d_{\kappa}$, that is, if a.e. $x \in E_{d_{\kappa}}$, we have  $\mathcal{M}(\dist{\cdot, E}^{-\gamma} \chi_{E_{d_{\kappa}}}) (x) \geq \dist{x, E}^{-\gamma}$. Let us then consider the case $x \notin E_{d_{\kappa}}$, that is, $d_{\kappa} \leq \dist{x, E}$. In this case we have
\begin{equation}\label{E:Delta:B:3:dxE}
E_{d_{\kappa}} \subset B(x, (2 + 1/\kappa) \, \dist{x, E}),
\end{equation}
since, given $y \in E_{d_{\kappa}}$, for $\eps > 0$ there exist $x_0, x_1 \in E$ with $|x-x_0| <  \dist{x, E} + \eps$ and $|y-x_1| < \dist{y, E} + \eps < d_{\kappa} + \eps$ and then
\begin{align*}
|y - x| & \leq |y - x_1| + |x_1 - x_0| + |x_0 -x| \leq d_{\kappa} +\diam{E} + \dist{x, E} + 2 \eps\\
& = d_{\kappa} + \frac{1}{\kappa} d_{\kappa} + \dist{x, E} + 2 \eps \leq (2 + 1/\kappa) \dist{x, E} + 2 \eps
\end{align*}
and \eqref{E:Delta:B:3:dxE} follows by letting $\eps \to 0$ (and by recalling that $E_{d_{\kappa}}$ is an open set). Next, define
\begin{equation}\label{g:kappa:L1:E}
C(\gamma, d_{\kappa}):= \int_{E_{d_{\kappa}}} \dist{z, E}^{-\gamma} dz < \infty
\end{equation}
so that by setting $B_x:= B(x, (2 + 1/\kappa) \, \dist{x, E})$ we get
$$
\frac{1}{|B_x|} \int_{B_x} \dist{z, E}^{-\gamma} \chi_{E_{d_\kappa}}(z) dz = \frac{1}{|B_x|} \int_{E_{d_\kappa}} \dist{z, E}^{-\gamma} dz  = \frac{C(\gamma, d_{\kappa})}{|B_x|}
$$
and then 
$$
\mathcal{M}(\dist{\cdot, E}^{-\gamma} \chi_{E_{d_\kappa}}) (x) \geq C(\gamma, d_{\kappa}) \omega_n^{-1} ((2 + 1/\kappa) \dist{x, E})^{-n}
$$ 
and \eqref{M:d:gamma:dk} is proved with $K:= (2 + 1/\kappa)^{-n}  \omega_n^{-1} C(\gamma, d_{\kappa})$. \end{proof}

\emph{Proof of Theorem \ref{thm:Hardy:ineq:p:q:WP:2}.} The proof will be a combination of Corollary~\ref{coro:Hardy:ineq:p:q:WP} and Lemma \ref{lemma:M:d:gamma:dk}. Fix $0 < \gamma < n - \overline{{\rm{dim}}}_M(E)$ and for $\kappa > 0$ let us use Corollary \ref{coro:Hardy:ineq:p:q:WP} with the function 
$$
g_\kappa:= \dist{\cdot, E}^{-\gamma} \chi_{E_{d_{\kappa}}}
$$
(which belongs to $L^1(\rn)$ due to \eqref{dist:gamma:R:L1}) where $d_\kappa:= \kappa \, \diam{E}$, so that by Lemma \ref{lemma:M:d:gamma:dk}  we have
$$
\mathcal{M}(g_\kappa)(x) \geq K \, \dist{x, E}^{-n}, 
$$
for every $x \in \rn$ with  $\dist{x, E} \geq d_{\kappa}$ where 
$$
K:= (2 + 1/\kappa)^{-n}  \omega_n^{-1} C(\gamma, d_{\kappa}) > \kappa^n \omega_n^{-1} C(\gamma, d_{\kappa})
$$
with $C(\gamma, d_{\kappa})$ being precisely $\norm{g_\kappa}{L^1(\rn)}$ by \eqref{g:kappa:L1:E}. Therefore, we get
$$
\mathcal{M}(g_\kappa)(x)^{q \Theta_{p,q}(1)/n} \geq \kappa^{q \Theta_{p,q}(1)} \omega_n^{-1} C(\gamma, d_{\kappa})^{q \Theta_{p,q}(1)/n} \, \dist{x, E}^{-q \Theta_{p,q}(1)} 
$$
for every $x \in \rn$ with  $\dist{x, E} \geq \kappa \diam{E}$ and the result follows from Corollary \ref{coro:Hardy:ineq:p:q:WP} as the finite quantity $\norm{g}{L^1(\rn)}^{\Theta_{p,q}(1)/n}$ appearing on the right-hand side of the inequality cancels out with the one on the left-hand side.\qed

\end{document}